\documentclass{amsart}

\usepackage{url}
\usepackage[cm]{fullpage}
\usepackage{amsthm}
\usepackage{lipsum}
\usepackage{amsfonts}
\usepackage{graphicx}
\usepackage{epstopdf}
\usepackage{algorithmic}
\ifpdf
  \DeclareGraphicsExtensions{.eps,.pdf,.png,.jpg}
\else
  \DeclareGraphicsExtensions{.eps}
\fi

\newcommand{\TheTitle}{Error identities for variational problems with obstacles}

\title{{\TheTitle}}

\author{S. Repin}
\address{V.A.Steklov Institute of Mathematics in St.-Petersburg, 191011, Fontanka 27, Sankt--Petersburg, Russia 
 and University of Jyv\"askyl\"a, P.O.Box 35, FI-40014,  Finland}
\email{repin@pdmi.ras.ru}
\urladdr{http://www.pdmi.ras.ru/\string~repin}

\author{J. Valdman}
\address{Institute of Mathematics and Biomathematics, Faculty of Science, University of South Bohemia, Brani\v sovsk\' a 31, CZ--37005 and Institute of Information Theory and Automation, Academy of Sciences, Pod vod\'{a}renskou v\v{e}\v{z}\'{\i}~4, CZ--18208~Praha~8, Czech Republic}
\email{jvaldman@prf.jcu.cz}
\urladdr{https://sites.google.com/site/janvaldman/}

\usepackage{amsopn}

\newtheorem{remark}{Remark}
\newtheorem{theorem}{Theorem}
\newtheorem{lemma}{Lemma}
\newtheorem{corollary}{Corollary}

\usepackage{epsfig,latexsym,amsfonts,amsmath,amssymb,color,mathrsfs,dsfont}
\usepackage{graphicx}
\usepackage{tikz}
\usepackage{mathtools}

\def\dvg{{\rm div}}

\def\IntO{\int\limits_{\Omega}}

\def\Rd{\Bbb R^d}

\def\wh{\widehat}

\newcommand\be{\begin{eqnarray*}}
\newcommand\ee{\end{eqnarray*}}
\newcommand\ben{\begin{eqnarray}}
\newcommand\een{\end{eqnarray}}

\def\W{{\sf W}}

\def\vz{v^*}
\def\pz{p^*}

\def\yz{y^*}

\def\MM{{\mathds M}}
\def\cP{{\mathcal P}}

\def\benn{\begin{eqnarray*}}
\def\eenn{\end{eqnarray*}}
\def\BR{{\mathbb R}}

\newcommand\dx{\, dx}

\newcommand{\om}{-}
\newcommand{\op}{+}
\newcommand{\opm}{\pm}

\newcommand{\Oum}{\Omega ^u_\om}
\newcommand{\Oup}{\Omega ^u_\op}
\newcommand{\Oun}{\Omega ^u_0}

\newcommand{\Ovm}{\Omega ^v_\om}
\newcommand{\Ovp}{\Omega ^v_\op}
\newcommand{\Ovn}{\Omega ^v_0}

\newcommand{\operm}[1]{(#1)_\om}
\newcommand{\operp}[1]{(#1)_\op}

\newcommand{\alpham}{\alpha_\om}
\newcommand{\alphap}{\alpha_\op}

\begin{document}


\maketitle

\begin{abstract}
The paper is concerned with a class of  nonlinear
  free boundary problems, which are usually
  solved by variational methods based on primal (or primal--dual) variational settings. We deduce and investigate special relations
  (error identities). They show
  that a certain nonlinear measure of the distance to the
  exact solution
  (specific for each problem) is equivalent to the respective  duality gap, which minimization is a keystone of all variational numerical methods. Therefore,  the identity defines the measure that contains 
  maximal quantitative information on the quality
  of a numerical solution
  available through these methods. The measure has
  quadratic terms generated by the linear
  part of the differential operator and nonlinear terms associated with free boundaries. We obtain fully computable two sided bounds
  of this measure and show that they provide efficient estimates
  of the distance between the
  minimizer and any function from the corresponding energy space. Several examples show that for different minimization sequence the balance
  between different components of the overall error measure may be different
  and domination of nonlinear terms may indicate that coincidence sets
  are approximated incorrectly.

 \end{abstract}



\section{Introduction}
Variational inequalities form an important class of nonlinear
models that describe free boundary phenomena arising in various
applied problems (see, e.g., G.~Duvaut and J.~L. Lions \cite{DuLi} and other publications cited therein). Usually free boundaries separate regions where  solutions possess quite different physical properties. Therefore, any reliable information on the shape and location of such a boundary is very important.
Qualitative properties of free boundaries are studied
by purely analytical (a priori) methods unlike quantitative
information, which in the vast
majority of cases can be obtained only by computational methods.
In this context, it is necessary to
know which quantitative information could be indeed extracted
from a numerical solution.

In this paper, we are concerned with two classes of variational
inequalities generated by obstacle type conditions.
Differentiability properties of exact solutions to these problems are, in general, restricted even if all external data of a problem are smooth (e.g., see the works of
H. Brezis \cite{Brezis}, L.A. Caffarelli \cite{Caff},
 D. Kinderlehrer and G. Stampacchia \cite{KS},
A. Friedman \cite{Friedman},
N. N. Uraltseva \cite{Uraltseva}).
In \cite{BS} it was proved that
there exists a unique solution  $u \in W^{2,2}(\Omega)$
of an obstacle problem
\be
\IntO\nabla u\cdot \nabla w dx\,\geq\,\IntO fw\,dx
\;\;\;\forall w\in
K:=\{w\in H^1(\Omega)\,\mid\,w=u_D\,on\,\partial\Omega,
\,\,w\geq\psi\}
\ee
if $\psi\in W^{2,2}(\Omega)$, $f\in L^2$, the function $u_D$ (which defines the Dirichl\'et boundary condition)
belongs to $W^{2,2}(\Omega)$ and satisfies the natural condition $u_D\geq \psi$
on $\partial\Omega$.

Many researches were focused
on clarifying mathematical properties of the coincidence set.
In particular, it was proved
that if the domain
$\Omega\subset R^2$ is strictly convex with a smooth boundary $\partial \Omega$
 and if the obstacle $\psi\in C^2(\Omega)$
is strictly concave, then the coincidence set is connected
and its boundary is smooth and homeomorphic to the unit circle (see, e.g. \cite{KS}).
However, in general, the structure of a coincidence set
can be very complicated and for any domain one can point
out such an obstacle that this set has any number
of disjoint subsets. 


Numerical methods for problems with obstacles (and many other
problems related to variational inequalities) were
systematically studied in R. Glowinski, J.-L. Lions, and R. Tremolieres
\cite{GlLiTr,Glowinski}. 
Getting the respective a priori rate convergence estimates (in terms
of the  mesh size $h$) was the first
question studied by many authors.
In the context
of finite element approximations such type estimates were
derived by 
R.~S.~Falk
\cite{Falk}
who proved  the standard a priori convergence
error estimates (with the rate $h$ for the $L^2$ norm of gradients
and the rate $h^2$ for the $L^2$ norm of the functions) 
provided that $u\in W^{2,2}$.
Convergence of mixed methods for problems with
obstacles was established in
F.~Brezzi, W. W. Hager and P. A. Raviart
\cite{BHR} and numerical methods based on the augmented Lagrangian
approach were studied in T. K\"arkk\"ainen, K. Kunisch, and P. Tarvainen \cite{KKT}.

This paper is concerned with other important questions arising
in  quantitative
analysis of nonlinear problems. One of them is {\em which measure $\mathds M$
of the distance to the exact solution is adequate (natural)
for a particular problem?} (see a discussion in  \cite{Repin2012}).
Furthermore, we must know which
properties of a solution are controlled by $\mathds M$ and
 deduce  explicitly computable bounds  (minorants and majorants). 
 In the paper, we study these questions in the context
 of obstacle type problems.
Our analysis is based upon general type {\it error identities} derived in \cite{NeRe,Repin1999,Repin2000} for a wide class
of convex variational problems. These identities establish equivalence
of a certain nonlinear measure  $\mathds M$ 
and the duality gap between the primal and dual energy functionals. 
 Since variational methods are based
 on minimization of this gap, the measure $\mathds M$ shows limits of quantitative analysis for this class of methods. 

For convenience of the reader we shortly recall the main
items necessary for understanding of the material.
Consider the class of variational problems
\ben
\label{1.1}
\inf\limits_{v\in V}J(v),\qquad J(v)=G(\Lambda v)+F(v),
\een
where $\Lambda:Y^*\rightarrow \BR$ is a bounded linear operator,
 $ G:Y\rightarrow \BR$ is a convex, coercive, and lower semicontinuous functional,
 $F:V\rightarrow\BR$ is another convex lower semicontinuous functional, and
 $Y$ and $V$ are reflexive Banach spaces. The dual spaces are denoted
 by $Y^*$ and $V^*$, respectively, and the duality pairings are denoted by $(y^*,y)$ and $\left<v^*,v\right>$.
 The dual variational problem consists of finding $p^*\in Y^*$
maximizing the dual functional 
\ben
I^*(y^*):=-G^*(y^*)-F^*(-\Lambda^*y^*) \label{dual_functional}
\een
 over the space $Y^*$. Here $G^*:Y^*\rightarrow \BR$ and $F^*:V^*\rightarrow \BR$
are the Young-Fenchel transforms (convex conjugates) of $G$ and $F$, respectively. Henceforth we use are the so called {\em compound} functionals
\be
&D_F(v,\vz):=&F(v)+F^*(\vz)- \left< \vz,v \right>, \\
& D_G(y,\;\yz)\;:=&G(y)+G^*(\yz)-(\yz,y)
\ee
 generated by the convex functionals
$F$ and $G$, respectively. These functionals are nonnegative
and vanish if and only if $v$ and $v^*$ (resp. $y$ and $y^*$)
are joined by special differential relations (see, e.g.,
\cite{EkTe}). Notice that in the simplest case where $V$ is a Hilbert space
and $F(v)=\frac12\|v\|^2$, the functional $D_F(v,\vz)$ coincides with the norm
$\frac12\|v-\vz\|^2$. However, in general $D_F(v,\vz)$ should be viewed
as a nonlinear measure, which vanishes if and only if the pair $(v,\vz)$
satisfies certain conditions.

 Let $\yz\in Y^*$ and $v\in  V$ be the functions compared with $\pz$ and $u$. Introduce the following (nonlinear)
measure of the distance between $\{u,\pz\}$ and
$\{v,\yz\}$:
\ben 
\label{1.7}
\MM(\{u,\pz\},\{v,\yz\}):=
D_F(u,-\Lambda^*\yz)+D_F(v,-\Lambda^*\pz)+D_G(\Lambda u,\yz)+D_G(\Lambda v,\pz)\,\geq\,0.
\een 
It vanishes
if and only if
\be
\Lambda v\in\partial G^*(p^*),\;
y^*\in\partial G(\Lambda u),\;
-\Lambda^*y^*\in\partial F(u),\;
v\in\partial F^*(-\Lambda^* p^*).
\ee
The above conditions are satisfied if and only
if $v=u$ and $y^*=p^*$ (i.e., if approximations coincide with the exact primal and dual solutions).
In
\cite{Repin1999}
 and \cite{NeRe} (Section 7.2), it was proved that
\begin{equation}
\label{erroridentity1}
\MM(\{u,\pz\},\{v,\yz\}) = J(v)-I^*(y^*).
\end{equation}
Hence
$\MM\{(u,\pz),(v,\yz)\}=0$
if and only if $J(v)=I^*(y^*)$ (what means that
$v$ is a minimizer
of the problem $\cP$ and $\yz$ is a maximizer
of the problem $\cP^*$).  

Two particular forms of (\ref{erroridentity1}) arise if we set $v=u$ or $\yz=\pz$. They are $\MM ( u, v ) := \MM (\{u,\pz\},\{v,\pz\})$ and $\MM (\pz, \yz) := \MM (\{u,\pz\},\{u,\yz\})$.
In view of (\ref{erroridentity1}),
\begin{eqnarray} 
\label{mainidentity}
&\MM ( u, v ) =&D_F(v,-\Lambda^*\pz)+D_G(\Lambda v,\pz)= J(v)-J(u),\\
\label{mainidentity_dual}
&\MM (\pz, \yz) =&D_F(u,-\Lambda^*\yz)+D_G(\Lambda u,\yz)= I^*(\pz)-I^*(\yz).
\end{eqnarray}
Numerical methods
are based either on minimization of the primal energy, or maximization of the dual energy, or on coupled minimization--maximization of both.  The identities
(\ref{erroridentity1}),
(\ref{mainidentity}), and
(\ref{mainidentity_dual}) show that
 the functional $\MM(\{u,\pz\},\{v,\yz\})$ 
(and its particular forms $\MM(u,v)$ and $\MM(\yz,\pz)$)
are in fact   the error
measures used by energy based numerical procedure designed to
 solve (\ref{1.1}). Since the error measures are equal to the respective duality
gaps, they present the strongest (and in a sense the most natural) measure
for the class of problems considered.

Below we study these identities for two classes of nonlinear variational problems
and show that they generated specific error measures containing two parts. The first part is presented
by a norm equivalent to  $H^1$ norm and the second one
is a nonlinear measure, which controls (in a rather weak sense) how accurately an approximate solution
recovers configuration of the free boundary.
We deduce directly computable quantities which majorate the
right hand sided of (\ref{erroridentity1}),
(\ref{mainidentity}), and
(\ref{mainidentity_dual}). Furthermore, 
we prove that the majorants  are sharp,
i.e., they do not contain an irremovable gap between the left
and right hand sides. The majorants possesses other important
properties, namely, they
need no a priori knowledge about the shape of a coincidence set,
 valid for any approximations
of the admissible functional (energy) set,
and  do not contain unknown (e.g., interpolation) constants. In the last section of the paper, we collect
computational results aimed to confirm theoretical analysis.
They are mainly focused on two points. First we show
that the measures correctly represent the quality of approximations for various minimizing sequences. 
Another observation is that
for different sequences different parts of the measure
may dominate, but their sum always correctly represent the error and can be efficiently estimated from above by the majorant.

\section{Classical obstacle problem}
\subsection{Variational setting}
We begin with the classical obstacle problem (see, e.g. \cite{Brezis,Friedman,KS}), where admissible functions belong to the set 
\be
&&K:=\{v\in  V_0:=H^1_0(\Omega) \,\mid\,\phi(x)\,\leq v(x)\, \leq \psi(x)\;
{\rm a.e.\,in\,}\Omega \}.
\ee
Here, $H^1_0(\Omega) $ denotes the Sobolev space of functions vanishing
on $\partial\Omega$ (hence we consider the case $u_D=0$),
 $\Omega\subset\mathbb{R}^d$ ($d \in \{1, 2, 3\}$) is a bounded domain with a Lipschitz continuous boundary $\partial\Omega$ and 
$\phi,\psi\in H^2(\Omega )$ are two given functions (lower and upper obstacles) such that
\be
&&\phi(x)\leq 0\;{\rm on}\;\partial\Omega,\quad
\psi(x)\geq 0\;{\rm on}\;\partial\Omega,
\quad\phi(x)\leq\psi(x),\quad\forall x\in \Omega.
\ee
The problem is to
find $u\in K$ satisfying the variational inequality
\ben
&&a(u,w-u)\geq\,
 ( f,w-u ) \quad \forall w \in K
\label{classical_obstacle}
\een
for a given function $f\in L^{2}(\Omega)$ and a bilinear form
 $$a(u,w):= \IntO A\nabla u\cdot\nabla w \dx.$$ 
It is assumed that $A $ is a 
symmetric matrix subject to the condition
\ben \label{A_assumptions}
A(x) \xi\cdot\xi\geq c_1 \, |\xi|^2\qquad c_1>0,\qquad \forall \xi \in
\mathbb{R}^{d}
\een
almost everywhere in $\Omega $. Under the assumptions made, the
unique solution $u \in K$ exists.
In general, the solution $u$ divides $\Omega $ into three sets:
\ben
&\Oum:=&\{x\in\Omega \,\mid\, u(x)=\phi(x) \}\,, \nonumber  \\
&\Oup:=&\{x\in\Omega \,\mid\, u(x)=\psi(x) \}\,, \label{coincidence_sets}    \\
&\Oun:=&\{x\in\Omega\,\mid\,\phi(x)<u(x)<\psi(x)\}\,. \nonumber 
\een
The sets  $\Oum$ and
$\Oup$  are the {\em lower} and
{\em upper coincidence sets}
and  $\Oun$
is an open set, where $u$ satisfies the Poisson equation 
$\dvg  (A \nabla u) + f=0$. Thus, the problem involves {\em free boundaries}, which are unknown a priori. Let $v$ be an approximation of $u$. It defines approximate  sets
\ben 
&\Ovm:=&\{x\in\Omega \,\mid\, v(x)=\phi(x)\}\,,  \nonumber \\
&\Ovp:=&\{x\in\Omega \,\mid\, v(x)=\psi(x)\}\,, \label{coincidence_sets_approximate}\\
&\Ovn:=&\{x\in\Omega\,\mid\,\phi(x)<v(x)<\psi(x)\} \nonumber \,.
\een
Notice that unlike the sets in (\ref{coincidence_sets}), the sets 
(\ref{coincidence_sets_approximate}) are known.

Solution of the problem (\ref{classical_obstacle}) can be represented in a mixed form, i.e., as a pair $(u, \pz)$, where the flux 
\ben \pz=A\nabla u \label{exact_flux}
\een
satisfies the conditions 
\ben
&\dvg \pz+f\leq 0\quad &{\rm on}\;\Oum, \nonumber \\
&\dvg \pz+f\geq 0\quad &{\rm on}\;\Oup, 
\label{equilibrium_of_p*}   \\ 
&\dvg \pz+f=0\quad &{\rm on}\;\Oun. \nonumber 
\een
The pair $(u, \pz)\in K\times L^2(\Omega,\Rd)$ is a saddle point of the respective minimax formulation. Under the above made
assumptions it exists. Moreover, $\pz$ has square summable divergence and satisfies the
relations \eqref{exact_flux} and \eqref{equilibrium_of_p*} almost everywhere in $\Omega$.


\subsection{Error measures}
The variational inequality \eqref{classical_obstacle} is known to have the equivalent form \eqref{1.1} for
\be
&\Lambda v=\nabla v,\qquad\qquad
&\Lambda^*\yz=-\dvg \yz,\\
&G(\Lambda v)=\frac12 \IntO  A\nabla v\cdot\nabla v \dx,\;\qquad\qquad
&F(v)=-\IntO fv\,dx +\chi_K(v),
\ee
where
$\chi_K$ is the characteristic functional of the set $K$, i.e.,
\be
\chi_K(v):=\left\{
\begin{array}{cc}
0\;&{\rm if}\;\phi\leq v\leq \psi,\\
+\infty& {\rm else}.
\end{array}
\right.
\ee
In this case, $V=V_0$, $V^*=H^{-1}(\Omega)$, 
$Y=L^2(\Omega,\Rd)$
\ben
G^*(\yz)&=& \frac12 \IntO A^{-1} \yz \cdot \yz \dx,   \label{G_dual}
\een
and
\ben
D_G(\Lambda v,\yz)&=&\frac12 \IntO (A\nabla v - \yz) \cdot (\nabla v  - A^{-1} \yz) \dx.  \label{DG}
\een
For $\yz=\pz$ and for $v=u$, we obtain
\ben
& D_G(\Lambda v,\pz)=\frac12\IntO A\nabla (u-v)\cdot\nabla (u-v) \dx =: \frac12 || \nabla (u - v) ||_A^2,  
\label{compound_G} \\
& D_G(\Lambda u,\yz)=\frac12\IntO A^{-1}(\pz - \yz) \cdot (\pz - \yz) \dx =: \frac12 ||  \pz - \yz ||_{A^{-1}}^2.  \label{compound_G_dual}
\een 
Next, for $v^*\in L^2(\Omega)$, 
\ben
F^*(v^*)=\sup\limits_{v\in K}\IntO v (v^*+f) \dx = \sup\limits_{v\in K}\IntO  ( - v (v^*+f)_\om + v (v^*+f)_\op  ) \dx 
= \IntO (-\phi(v^*+f)_\om+\psi(v^*+f)_\op) \dx. 
\label{F_dual}
\een
Here, $(z)_\om$ and $(z)_\op$ denote the negative and positive parts of the quantity $z$, i.e.,
$(z)_\om :=-\min\{0, z\}, (z)_\op :=\max\{0, z \}$. They
satisfy the relations
$z=  -  (z)_\om  + (z)_\op$ and $|z| =  (z)_\om + (z)_\op$. 

In view of (\ref{F_dual}), we deduce explicit form
of the functional $D_F$ provided that
$y^*\in Y^*_{\rm div}(\Omega):=\left\{y^*\in Y^*\,\mid\,
\dvg y^*\in L^2(\Omega)\right\}$:
\begin{multline}
D_F(v,-\Lambda^*\yz)= F(v)+F^*(-\Lambda^*\yz) + \left< \Lambda^*\yz,v \right>= \\
=\IntO(-fv -\phi \, \operm{ \dvg \yz+f } + \psi \, \operp{ \dvg \yz+f }  - \dvg \yz v) \dx =\\
=\IntO((v-\phi) \, \operm{ \dvg \yz+f } + (\psi - v) \, \operp{ \dvg \yz+f} ) \dx.
\end{multline}
Since $\pz$ belongs to $Y^*_{\rm div}(\Omega)$ 
and satisfies the relation \eqref{equilibrium_of_p*},
we find that
\begin{multline}\label{measure_obstacle}
D_F(v,-\Lambda^*\pz) =  -\int\limits_{\Oum}(v-\phi)(\dvg \pz+f) \dx+ \int\limits_{\Oup}(\psi-v)(\dvg \pz +f) \dx =  \\
= -\int\limits_{\Oum}(v-\phi)(\dvg A\nabla \phi+f) \dx+  \int\limits_{\Oup}(\psi-v)(\dvg A\nabla \psi+f) \dx.
\end{multline}
This quantity can be viewed as a certain  measure
\ben
\mu_{\phi\psi}(v) :=\!\int\limits_{\Oum}\!\! \W_\phi (v-\phi)\dx+  \int\limits_{\Oup}\!\! \W_\psi(\psi-v)\dx,
\label{measure_primal_classical}
\een
where
$
\W_\phi:=-(\dvg A\nabla \phi+f)$, 
$\W_\psi:=\dvg A\nabla \psi+f$
are two nonnegative 
weight functions generated by the source term $f$, the obstacles $\psi, \phi$ and the diffusion $A$. 
It is clear that $\mu_{\phi\psi}(v)=0$ if $\Ovm\subset \Oum $ and $\Ovp\subset \Oup $. In other words,
if all points of approximate sets $\Ovm$ and $\Ovp$ indeed
belong to the coincidence sets, then the measure is zero.
\begin{remark} \label{exampleI}
Assume that  $A={\mathbb I}$ (the identity matrix), obstacles $\phi, \psi$ are harmonic functions ($\triangle \phi=\triangle \psi = 0$ in $\Omega$) satisfying $\phi < 0 < \psi$ almost everywhere in $\Omega$ and $f=const \not = 0$. If $f >0$ then $\Oum=\emptyset$ (the lower obstacle $\phi$ is never active) and
\ben \label{measure_example_1}
\mu_{\phi\psi}(v)=f
 \int\limits_{\Oup} (\psi-v)\dx=f\|\psi-v\|_{L^1(\Oup)} = f\|\psi-v\|_{L^1(\Oup \setminus \Ovp)}.
\een
Here, we decomposed $$\Oup = (\Oup \setminus \Ovp) \cup (\Oup \cap \Ovp) $$ and applied the equality
$\|\psi-v\|_{L^1(\Oup)}=\|\psi-v\|_{L^1(\Oup \setminus \Ovp)}$ (which holds because $\|\psi-v\|_{L^1(\Oup \cap \Ovp)}={0}$).
Analogously, if $f <0$ then $\Oup=\emptyset$ (the upper obstacle $\psi$ is neven active) and
\ben \label{measure_example_2}
\mu_{\phi\psi}(v)=-f
 \int\limits_{\Oum} (v-\phi)\dx=-f\|v-\phi\|_{L^1(\Oum)} = -f\|v-\phi\|_{L^1(\Oum \cap \Ovm)}.
\een
We see that $\mu_{\phi\psi}(v)$ represents a certain measure, which
controls (in a weak integral sense) whether or not the function $v$
coincides with obstacles $\psi, \phi$ on true coincidence sets $\Oum$
and $\Oup$. 
\end{remark}

Analogously, the quantity
\ben
\label{measure_obstacle_dual}
&& D_F(u,-\Lambda^*\yz) = - \int\limits_{\Omega^{\yz}_\om}(u-\phi)(\dvg \yz+f) \dx+ \int\limits_{\Omega^{\yz}_\op}(\psi-u)(\dvg \yz +f) \dx 
\een
forms another measure
\ben
\mu_{\phi\psi}^{*}(\yz) := -\int\limits_{\Omega^{\yz}_\om }\!\!(u-\phi)(\dvg \yz+f) \dx+ \int\limits_{\Omega^{\yz}_\op }\!\!(\psi-u)(\dvg \yz +f) \dx,
\label{measure_dual_classical}
\een
where the sets
\ben
&\Omega ^{\yz}_\om:=\{x\in\Omega \,\mid\, \dvg \yz +f  < 0\}\,, \nonumber  \\
&\Omega ^{\yz}_\op:=\{x\in\Omega \,\mid\, \dvg \yz +f  > 0 \}\,, \label{coincidence_sets_approximate_flux}    \\
&\Omega ^{\yz}_0:=\{x\in\Omega\,\mid\,\dvg \yz +f  = 0 \}\,\nonumber 
\een
are approximations of $\Omega_\om$, $\Omega_0$, and $\Oum, \Oup, \Oun$ generated on the basis of dual solution $\yz$.
It is clear that this measure is zero if $\Omega^{\yz}_\om \subset \Oum $ and $\Omega^{\yz}_\op \subset \Oup $.
Hence, the measure $\mu_{\phi\psi}^{*}(\yz)$ is positive
if the sets $\Omega^{\yz}_\om$ and  $\Omega^{\yz}_\op$
contain parts which do not belong to true coincidence sets. We summarize
properties of $\mu_{\phi\psi}(v)$ and $\mu_{\phi\psi}^{*}(\yz)$ as follows:
\ben \label{zero_measure_case}
&\Oum \subset \Ovm \;{\rm and}\; \Oup \subset \Ovp \quad \Rightarrow \quad &\mu_{\phi\psi}(v)=0,\\
 \label{zero_dual_measure_case}
&\Omega^{\yz}_\om\subset \Oum, \;{\rm and}\; \Omega^{\yz}_\op\subset \Oup \quad \Rightarrow \quad& \mu_{\phi\psi}^{*}(\yz)=0.
\een

Now we use \eqref{erroridentity1}, \eqref{mainidentity}, and  \eqref{mainidentity_dual}  and deduce
error identities for the obstacle problem.

\begin{theorem}[energy identities for the classical obstacle problem] \label{theorem1}
Let $v$ and $\yz$ be approximations of $u$ and $\pz$, respectively. Then, 
\ben
\label{id:primal}
&\MM (u,v)=&\frac12 || \nabla (u - v) ||_A^2 + \mu_{\phi\psi}(v) = J(v) - J(u), \label{MM_final_obstacle} \\
&\MM (\pz,\yz)=&\frac12 ||  \pz - \yz ||_{A^{-1}}^2 +  \mu_{\phi\psi}^{*}(\yz) = I^*(\pz) - I^*(\yz). 
\label{id:dual}
\een 
\end{theorem}
Theorem \ref{theorem1} establishes exact error identities for the classical obstacle problem in terms the primal and dual posings. 
In view of the relation between  the primal
and dual functionals, the identities (\ref{id:primal})
and (\ref{id:dual}) yield
\ben
&&\MM(\{u,\pz\},\{v,\yz\})= \MM (u,v) +  \MM (\pz,\yz) = J(v) - I^*(\yz). 
\label{id:primaldual} 
\een
This error identity holds for the mixed nonlinear measure $\MM(\{u,\pz\},\{v,\yz\})$ (which decomposes additively to two primal nonlinear measures). It shows that the duality gap
consists of four nonnegative quantities. Two of them are quadratic terms associated with energy errors. Two others
are nonlinear measures $\mu_{\phi\psi}(v)$ and $\mu_{\phi\psi}^{*}(\yz)$ defined by \eqref{measure_primal_classical} and \eqref{measure_dual_classical} 
Without taking them into account, only inequalities 
$$\frac12 || \nabla (u - v) ||_A^2 \leq J(v) - J(u), \qquad \frac12 ||  \pz - \yz ||_{A^{-1}}^2 \leq I^*(\pz) - I^*(\yz) $$ can be obtained. 
\subsection{Computable bounds of error measures}
First we show that the measure
 $\MM(\{u,\pz\},\{v,\yz\})$ can be directly computed for any pair of approximate
 solutions $\{v,\yz\}$ provided that $\yz$ possesses an additional regularity. 

\begin{theorem}
 \label{theorem2}
Let $\{v,\yz\}\in K\times Y^*_{\dvg}(\Omega)$. Then, 
\ben \label{error_measure_obstacle_computable}
\MM(\{u,\pz\},\{v,\yz\})= \frac12 ||A \nabla v - \yz ||_{A^{-1}}^2 +  \Upsilon(v,\yz),
\een
where
\ben  \label{majorant_nonlinear_classical}
\Upsilon(v,\yz) := \int\limits_{\Omega^{\yz}_\om \setminus \Omega^{v}_\om} (\phi-v) (\dvg \yz+f) \dx
+ \int\limits_{\Omega^{\yz}_\op \setminus \Omega^{v}_\op } (\psi - v) (\dvg \yz+f) \dx. 
\een
\begin{proof}
In view of  \eqref{dual_functional}), \eqref{G_dual}, and \eqref{F_dual}, we have
\ben
J(v) &=& \frac12 || \nabla v ||_A^2 - \IntO fv \dx, \\
I^*(\yz) &=& - \frac12 || \yz ||_{A^{-1}}^2 + \IntO ( \phi \, \operm{\dvg \yz +f} - \psi \, \operp{\dvg \yz +f } ) \dx.
\een
According to \eqref{id:primaldual},
\begin{multline} \label{error_measure_obstacle_computable_intermediate}
\MM(\{u,\pz\},\{v,\yz\})= 
\frac12 ||\nabla v||_A^2 - \IntO fv \dx + \frac12 ||\yz ||_{A^{-1}}^2 + \\ 
+ \IntO ( -\phi \, \operm{ \dvg \yz+f } + \psi \, \operp{ \dvg \yz+f }   )\dx.
\end{multline}
Since
\begin{multline*}
 \frac12 ||\nabla v||_A^2  + \frac12 ||\yz ||_{A^{-1}}^2 - \IntO fv \dx 
= \frac12 ||A \nabla v - \yz ||_{A^{-1}}^2 + \IntO (\yz \cdot \nabla v - fv)\dx = \\
= \frac12 ||A \nabla v - \yz ||_{A^{-1}}^2 - \IntO (\dvg \yz + f) v \dx
\end{multline*}
and 
\begin{multline*}
- \IntO (\dvg \yz + f) v \dx 
+ \IntO ( -\phi \, \operm{ \dvg \yz+f }  + \psi \, \operp{ \dvg \yz+f } ) \dx = \\
= - \int\limits_{ \Omega^{\yz}_\om} (v - \phi) (\dvg \yz+f) \dx + \int\limits_{\Omega^{\yz}_\op} (\psi - v) (\dvg \yz+f) \dx  = \\
= - \int\limits_{ \Omega^{\yz}_\om \setminus \Ovm} (v - \phi) (\dvg \yz+f) \dx + \int\limits_{\Omega^{\yz}_\op \setminus \Ovp } (\psi - v) (\dvg \yz+f) \dx 
\end{multline*}
the substitution of last two equalities in \eqref{error_measure_obstacle_computable_intermediate} yields
 \eqref{error_measure_obstacle_computable}.
\end{proof}
\end{theorem}
\begin{remark}
Assume that the right hand side of (\ref{error_measure_obstacle_computable})
is equal to zero. Then $\yz=A\nabla v$ and
\be
v=\phi \qquad {\rm if}\;\dvg\yz+f<0,\\
v=\psi \qquad {\rm if}\;\dvg\yz+f>0.
\ee
Hence, $\Omega^{\yz}_\om\subset\Ovm$ and $\Omega^{\yz}_\op\subset\Ovp$. The sets $\Ovp$ and $\Ovm$ do not intersect as well as the sets
$\Omega^{\yz}_\op$ and $\Omega^{\yz}_\om$. Therefore, the set $\Omega^v_0=\Omega\setminus(\Ovp\cup\Ovm)$ is contained in the set
$\Omega^{\yz}_0=\Omega\setminus(\Omega^{\yz}_\op\cup \Omega^{\yz}_\om)$.
Thus, $\dvg\yz+f=0$ in $\Omega^v_0$.
For any $w\in K$, we have
\be
\IntO A\nabla v\cdot\nabla (w-v) \dx-\IntO f(w-v) \dx=
\int\limits_{\Ovm}(\dvg \yz+f)(\phi-w) \dx\\
+\int\limits_{\Ovp}(\dvg \yz+f)(\psi-w) \dx+\int\limits_{\Omega^v_0}(\dvg \yz+f)(v-w) \dx.
\ee
The right hand side of the above relation is nonnegative.
Indeed, the first two integrals are nonnegative and the last one is equal to zero.
This means that $v$ satisfies the variational inequality
and, consequently, the pair $\{v,\yz\}$ coincides with $\{u,\pz\}$.
\end{remark}
\begin{remark}\label{remark_error_measure_obstacle_computable_simplified}                   
If approximations of the coincidence sets (constructed on the basis
of $v$ and $\yz$) satisfy the relations
$ \Omega^{\yz}_\om\subset \Ovm$ and $\Omega^{\yz}_\op\subset \Ovp $,
then (\ref{error_measure_obstacle_computable}) reads
\ben
\label{error_measure_obstacle_computable_simple}
 \MM(\{u,\pz\},\{v,\yz\}) =\frac12 ||A \nabla v - \yz ||_{A^{-1}}^2. 
 \een
Moreover, if
$\Omega^{\yz}_\om \subset \Oum \subset \Ovm$ and  $\Omega^{\yz}_\op\subset \Oup\subset \Ovp,$
then nonlinear terms of $\MM(\{u,\pz\},\{v,\yz\})$ vanish and we arrive at the equality
\be
|| \nabla (u - v) ||_A^2 + ||  \pz - \yz ||_{A^{-1}}^2  = ||A \nabla v - \yz ||_{A^{-1}}^2.
\ee
However, the sets $\Oum$ and $\Oup$ are unknown, so that in practice it is impossible to verify  the conditions that yield this simplest (hypercircle type) form of the error identity.
\end{remark}

Theorem \ref{theorem2} provides a way to compute $ \MM(\{u,\pz\},\{v,\yz\})$,
which is the sum of
error measures $\MM (u,v)$ and $\MM (\pz,\yz)$. These measures separately
evaluate deviations of $v$ from $u$ and $\yz$ from $\pz$. It is desirable to have
guaranteed bounds for them as well (notice that in view of (\ref{error_measure_obstacle_computable}) two sided bounds of $\MM (u,v)$
imply two sided bounds of $\MM (\pz,\yz)$ and vise versa). For this purpose, we require knowledge of the exact energy $J(u)$ (or $I^*(\pz)$), which is generally
unknown. However, their is a way to derive computable bounds of  $\MM (u,v)$
without this knowledge (see \cite{ReGruyter, Repin2003}). In this section, we briefly discuss some of them addressing the reader to
a more systematic exposition and numerical
tests to the above cited literature and  \cite{NeRe}. 

The first bound of $\MM (u,v)$ has the form
\begin{multline}
\label{upper_bound}
\MM (u,v) \,\leq\, \MM^+(v; \beta, \lambda_1, \lambda_2, \yz)
:=(1+\beta^{-1})D_G(\nabla v,\yz) \\
+\frac12 C^2_\Omega(1+\beta)\| \dvg \yz+f+\lambda_1-\lambda_2\|^2 _\Omega
+\IntO (\lambda_1(v-\phi)+\lambda_2(\psi-v))\,dx.
\end{multline}
The majorant $\MM^+$ contains contains free variables:  $\beta > 0$,  $\yz\in Y^*_{\dvg}(\Omega)$, and two nonnegative functions (Lagrange multipliers) $\lambda_1, \lambda_2 \in L^2(\Omega)$. The constant $C_\Omega > 0$ is 
a minimal constant in a Friedrichs type inequality
\ben
\label{friedrichs}
\|w\|\,\leq\,C_\Omega \|\nabla w\|_A \qquad \mbox{for all } w \in V_0.
\een
It is not difficult to show that for any $v$, there exist
$\beta$, $\lambda_1$, $\lambda_2$, and  $\yz$ such that
 \eqref{upper_bound} holds as the equality.
Indeed, set $\yz=\pz$,
and 
\ben
&&\lambda_1=-(\dvg \pz+f),\quad \lambda_2=0 \quad {\rm on}\;\Oum, \nonumber \\
&&\lambda_2=\dvg \pz+f,\quad \lambda_1=0\quad {\rm on}\;\Oup, \label{exact_choice_lambda} \\
&&\lambda_1=0, \quad \lambda_2=0\qquad {\rm on}\;\Oun \nonumber.
\een
Then, the second term of $\MM^+$ vanishes (for any choice of $\beta$) and the third term is equal to $\mu_{\phi\psi}(v)$. By taking a limit $\beta \rightarrow +\infty$, the first term 
converges to 
\be
D_G(\nabla v,\pz)=
\IntO\left(\frac{1}{2}A\nabla v\cdot \nabla v+\frac{1}{2}A^{-1}\pz\cdot\pz-\nabla v\cdot\pz\right)dx=\frac12\|\nabla(v-u)\|^2_A.
\ee

The choice \eqref{exact_choice_lambda} of Lagrange multipliers  is theoretically important since it depends on the exact solution $u$. It is replaced by different choices in practical computations. 
If we set alternatively
\ben
&&\lambda_1= \operm{ \dvg \yz+f },\quad \lambda_2=0\quad {\rm on}\;\Ovm, \nonumber \\
&&\lambda_2= \operp{ \dvg \yz+f },\quad \lambda_1=0\quad {\rm on}\;\Oup, \label{easy_choice_lambda}\\
&&\lambda_1=0,\quad \lambda_2=0 \qquad {\rm on}\;\Ovn \nonumber ,
\een
the third term of \eqref{upper_bound} vanishes and we obtain another majorant (which is free of  $\lambda_1, \lambda_2$)
\begin{equation} \label{majorant_obstacle_simplified}
\MM^+_1(v; \beta, \yz)=\frac12(1+\beta^{-1})D_G(\nabla v,\yz)
+\frac12 C^2_\Omega(1+\beta)\| [f+\dvg \yz]_v\|^2,
\end{equation}
where 
\ben 
[f+\dvg \yz]_v:=
\left\{
\begin{array}{ll}
\operm{ f+\dvg \yz} \quad& \mbox{in } \Ovp, \\
\operp{ f+\dvg  \yz} \quad& \mbox{in } \Ovm,  \\
f+\dvg \yz\quad&\mbox{in } \Ovn.
\end{array}
\right.
\een
More accurate optimization of \eqref{upper_bound} with respect to $\lambda_1, \lambda_2$ provides a sharper majorant \cite{Repin2003}  in the form 
\begin{equation} \label{majorant_obstacle_simplified_sharper}
\MM^+_2(v; \beta, \yz):=\frac12(1+\beta^{-1})D_G(\nabla v,\yz)+\frac12  \IntO R(v,f+\dvg  \yz,\beta)  \,dx ,
\end{equation}
where
\be
R(v,r,\beta):=
\left\{
\begin{array}{ll}
\frac{-(\phi -v)^2}{c_{\beta}} + 2 r (\phi - v) \quad&\mbox{if } c_{\beta} r + v \leq \phi,\\
\frac{-(\psi -v)^2}{c_{\beta}} + 2 r (\psi - v) \quad&\mbox{if } c_{\beta} r + v \geq \psi,\qquad c_{\beta}=C^2_\Omega(1+\beta)\\
c_{\beta} r^2 \quad&\mbox{if } \phi < c_{\beta} r + v < \psi.
\end{array}
\right.
\ee
Practical computations of majorants for the classical obstacle problem are further explained in \cite{BuRe,HaVa,HaVa2}.

%

\begin{remark}
Since 
$J(v)-J(w) \geq J(v) - J(u)$ holds for all $w \in K$, we always have a computable lower bound 
\begin{equation}
\label{lower_bound}
J(v)-J(w)=:\MM^-(v,w) \leq \MM (u,v).
\end{equation}
In practice, a suitable $w$ can be constructed by local (e.g., patch wise) improvement
of $v$ and ideas of hierarchical basis methods.
\end{remark}


\section{Double obstacle problem}
\subsection{Variational setting}
The following double-obstacle problem (also known as the two--phase obstacle problem), was studied in 
H. Shahgholian, N. N. Uraltseva, and G. S.Weiss \cite{ShahgolianUraltsevaWeiss}, 
N.N. Uraltseva \cite{Uraltseva2001},
G. S. Weiss \cite{Weiss} and some other papers cited therein. Here the variational (energy) functional $J(v)$ is defined by the relation 
\begin{equation}
J(v) := \IntO \Big( \frac{1}{2} A \nabla v\cdot\nabla v  - f v + \alpha_{\op} \operp{ v } + \alpha_{\om} \operm{ v }  \Big) \dx.    \label{eq:energy}
\end{equation}
The functional $J(v)$  is minimized on the set
$$
V_0+u_D:=\{v=v_0+u_D\,:\, v_0 \in V_0,\;
u_D \in H^1(\Omega)\}. 
$$
Here $u_D$ is a given bounded function that defines the boundary condition
($u_D$   may attain both positive and negative values on different parts of the boundary $\partial \Omega$). It is assumed that  the coefficients $\alpha_{\op},
\alpha_{\om}: \Omega \rightarrow \mathbb{R}$ are  positive
constants (without essential difficulties the consideration and main results can be extended to the case where they are positive
Lipschitz continuous functions). Also, it is assumed that
 $f \in L^{\infty}(\Omega)$, $A \in L^{\infty}(\Omega, \mathbb{R}^{d \times d})$, 
 and  the condition \eqref{A_assumptions} holds. 
 Since the functional $J(v)$ is strictly
convex and continuous  on $V$, existence and uniqueness of a minimizer $u \in K$ is guaranteed by well known results
of the calculus of variations (see, e.g., \cite{EkTe}).  Analysis of the corresponding 
Euler-Lagrangian equation leads to  the nonlinear problem
(\cite{ShahgolianUraltsevaWeiss,Weiss,Uraltseva2001})
\begin{equation}
{\rm{div}} (A \nabla u) + f = \alpha_{\op}\,\chi_{\{u>0 \}} -
\alpha_{\om} \,\chi_{\{u<0 \}},
 \qquad u=u_D\;{\rm on}\;\partial\Omega,
 \label{euler-lagrange}
 \end{equation}
where $\chi$ denotes the characteristic function of a set (attaining values 1 and 0 inside and outside the set, respectively).
A physical interpretation of the problem \eqref{euler-lagrange} is presented by
an elastic membrane touching the planar phase boundary between two
liquid/gaseous phases (see, e.g., \cite{ShahgolianUraltsevaWeiss}). 

We introduce two decompositions
of $\Omega$ associated with the minimizer $u$
and an approximation $v$:
\ben
&\Oum:=\{x\in \Omega \,\mid\, u(x)<0\}, \nonumber  \\
&\Oup:=\{x\in \Omega \,\mid\, u(x)>0\},  \label{coincidence_sets_double}\\
&\Oun:=\{x\in \Omega \,\mid\, u(x)=0\}, \nonumber
\label{Omega_u_double_obstacle}
\een
and
\ben 
&\Ovm:=\{x\in \Omega \,\mid\, v(x)<0\}, \nonumber\\
&\Ovp:=\{x\in \Omega \,\mid\, v(x)>0\}, \label{coincidence_sets_approximate_double}\ \\
&\Ovn:=\{x\in \Omega \,\mid\, v(x)=0\}. \nonumber
\label{Omega_v_double_obstacle}
\een
These decompositions generate exact and approximate free boundaries. Using the above notation we can rewrite
 \eqref{euler-lagrange} as follows
\ben \label{equilibrium_double}
\dvg (A\nabla u) +f = 
\begin{cases}
\, \alpha_{\op} \,\qquad \mbox{in } \Oup, \\
\, -\alpha_{\om} \, \quad \mbox{in } \Oum,\\
\, 0 \quad \,\qquad \mbox{in } \Oun.
\end{cases}
\een 
\subsection{Error measures}
The problem is reduced to  \eqref{1.1} if $V=V_0:=H^1_0(\Omega)$, $Y=L^2(\Omega,\Rd)$, $\Lambda w=\nabla w$,
$\Lambda^*y^*=-\dvg y^*$, and the functionals
\be
&&\wh G(y)=\frac12\IntO  A(y+y_D)\cdot (y+y_D)\,dx,\quad y_D=\nabla u_D,\\
&&\wh F(v_0):=   \IntO \Big(-f (v_0+u_D) + \alpha_{\op} \operp{ v_0+u_D} + \alpha_{\om} \operm{ v_0+u_D }  \Big) \dx
\ee
stand for $G$ and $F$, respectively.
The problem is to find $u_0\in V_0$ such that the functional
$\wh J(v_0)=\wh G(\nabla v_0)+\wh F(v_0)$ attains infimum on the space $V_0$.
\begin{multline}
\label{Gz}
\wh G^*(y^*)=\sup\limits_{y\in Y}\IntO \left(\yz\cdot y-\frac12 A(y+y_D)\cdot (y+y_D)\right) \dx\\
=
\sup\limits_{y\in Y}\IntO \left(\yz\cdot(y-y_D)-\frac12 Ay\cdot y \right) \dx=
\IntO \left(\frac12 A^{-1}\yz\cdot\yz-\yz\cdot y_D \right) \dx
\end{multline}
Hence, 
\ben
\label{DwhG}
D_{\wh G}(\Lambda v_0,\yz)=\IntO\left(\frac12 A\nabla  v\cdot\nabla  v+\frac12A^{-1}\yz\cdot\yz-\yz\cdot \nabla v \right) \dx 
=\frac12\|A\nabla v-\yz\|^2_{A^{-1}},
\een
for any $ v=v_0+u_D$.
Computation of  $\wh F^*(\vz)$ is more sophisticated.
\begin{lemma}
Let $\vz\in L^\infty(\Omega)$. Then,
\ben
\label{Fz_double_obstacle}
\wh F^*(\vz)=\left\{
\begin{array}{cc}
-\IntO \vz u_D \dx           &  \mbox{if }\; v^*+f \in [- \alpha_{\om},  \alpha_{\op}], \\
+\infty      & \mbox{else}.
\end{array}
\right.
\een
\end{lemma}
\begin{proof}
Assume that $\vz+f>\alpha_+$ on some open subset $\omega\subset\Omega$.
Then this inequality holds on a ball $B\subset\omega$. Define two smooth cut off functions
$\lambda^\epsilon_1$ and $\lambda^\epsilon_2$ such that
\be
& \lambda^\epsilon_i(x)\in [0,1],\qquad & i=1,2,\\
&\lambda^\epsilon_1=1\;{\rm on}\;\partial\Omega,\qquad &\lambda^\epsilon_1=0\;{\rm if}\;
{\rm dist}(x,\partial\Omega)>\epsilon,\\
&\lambda^\epsilon_2=1\;{\rm in}\;B,\qquad &\lambda^\epsilon_2=0\;{\rm if}\;
{\rm dist}(x,B)>\epsilon,\quad {\rm supp}\lambda^\epsilon_2\subset\omega.
\ee
Here $\epsilon$ is a positive quantity smaller than $\frac12{\rm dist}(B,\partial\Omega)$. For any $\rho\in {\mathbb R}$, the function 
$v^\epsilon:=\lambda^\epsilon_1 u_D+\rho \lambda^\epsilon_2$ belongs to
$V_0+u_D$.
It is not difficult to see that
\be
v^\epsilon=\left\{
\begin{array}{cc}
\lambda^\epsilon_1 u_D              &{\rm in}\;  S^\epsilon_1:={\rm supp}\lambda^\epsilon_1, \\
\rho\lambda^\epsilon_2           & {\rm in}\;   S^\epsilon_2:={\rm supp}\lambda^\epsilon_2\setminus B, \\
\rho  &{\rm in}\; B,\\
0    & \mbox{in\,\
all\,other\,points}
\end{array}
\right.
\ee
and
\be
(v^\epsilon)_-=\left\{
\begin{array}{cc}
\lambda^\epsilon_1 (u_D)_-              &{\rm in}\;  S^\epsilon_1, \\
0    & \mbox{in\,\
all\,other\,points}
\end{array}
\right. ,
\qquad
(v^\epsilon)_+=\left\{
\begin{array}{cc}
\lambda^\epsilon_1 (u_D)_+              &{\rm in}\;  S^\epsilon_1, \\
\rho\lambda^\epsilon_2           & {\rm in}\;   S^\epsilon_2, \\
\rho  &{\rm in}\; B,\\
0    & \mbox{in\,\
all\,other\,points}.
\end{array}
\right.
\ee
Therefore,
\ben
\wh F^*(\vz)&=&\sup\limits_{v_0\in V_0}
\left\{\IntO (\vz v_0+f (v_0+u_D)-\alpha_-(v_0+u_D)_--\alpha_+(v_0+u_D)_+)\dx \right\} \nonumber \\
&=&\sup\limits_{v\in V_0+u_D}
\left\{\IntO ((\vz+f) v-\alpha_-(v)_-    -\alpha_+(v)_+) \dx\right\}-\IntO \vz u_D \dx   \nonumber \\
&&\geq
\IntO ((\vz+f) v^\epsilon-\alpha_-(v^\epsilon)_-    -\alpha_+(v^\epsilon)_+) \dx-\IntO \vz u_D \dx  \label{Fz}\\
&&\quad =\int\limits_{S^\epsilon_1}((\vz+f)\lambda^\epsilon_1 u_D-\alpha_-(u_D)_--\alpha_+(u_D)_+) \dx  \nonumber \\
&&\qquad +\int\limits_{S^\epsilon_2} \rho(\vz+f-\alpha_+)\lambda^\epsilon_2 \dx+
\rho \int\limits_{S^\epsilon_2} (\vz+f-\alpha_+) \dx-\IntO \vz u_D \dx. \nonumber
\een
Let $\epsilon\rightarrow 0$ and $\rho\rightarrow+\infty$. Then the first integral in the right
hand side vanishes, the second is positive and the third tends to $+\infty$.
Hence, $F^*(\vz)=+\infty$. 

Quite analogously we prove that $F^*(\vz)=+\infty$ if $\vz+f<\alpha_-$
on some open set $\omega\subset\Omega$.
It remains to show that $F^*(\vz)=-\IntO \vz u_D\,dx$ if $-\alpha_-\leq \vz+f\leq\alpha_+$.
For this purpose, we define $v^\epsilon:=\lambda^\epsilon_1 u_D$.
In this case,
\begin{multline}
\label{Fzo}
\wh F^*(\vz)=
\sup\limits_{v\in V_0+u_D}
\left\{\IntO ((\vz+f) v-\alpha_-(v)_-    -\alpha_+(v)_+)dx\right\}-\IntO \vz u_D \dx\\
=
\int\limits_{\Omega^v_-} ((\vz+f +\alpha_-) vdx+\int\limits_{\Omega^v_+}(\vz+f-\alpha_+)v \dx-\IntO \vz u_D \dx.
\end{multline}
We see that the first two integrals are nonpositive, so that
$\wh F^*(\vz)\leq -\IntO \vz u_D \dx$. On the other hand,
\be
\int\limits_{\Omega^{v^\epsilon}_-} ((\vz+f +\alpha_-) v^\epsilon \dx+\int\limits_{\Omega^{v^\epsilon}_+}(\vz+f-\alpha_+)v^\epsilon \dx\rightarrow 0
\ee
as $\epsilon\rightarrow 0$ and we arrive at 
(\ref{Fz_double_obstacle}).
\end{proof}
\begin{corollary}
If $\vz$ satisfies $-\alpha_-\leq \vz+f\leq\alpha_+$, then
\be
D_{\wh F}(v_0)={\wh F}(v_0)+{\wh F}^*(\vz)-<\vz,v_0>
=\IntO \Big(-(f+\vz) v + \alpha_{\op} \operp{ v} + \alpha_{\om} \operm{ v }  \Big) \dx,
\ee
where $v=v_0+u_D$.
Hence, if
\be
\yz\in Y^*_{\dvg,[-\alpha_-,\alpha_+]}:=
\Bigl\{\yz\in Y^*: \dvg\yz+f\in [-\alpha_-,\alpha_+]\;a.e.\,in\;\Omega\Bigr\},
\ee
then
\begin{multline}\label{compound_F_double}
D_{\wh F}(v_0,-\Lambda^*\yz)
=\IntO \left(\alpha_{\op} \operp{ v } + \alpha_{\om} \operm {v } -( \dvg \yz +f)v \right) \dx = \\
= \int\limits_{\Ovp}\left(\alpha_{\op}-( \dvg \yz+f)\right)v \dx + \int\limits_{\Ovm}\left(-\alpha_{\om}-( \dvg \yz+f)\right)v \dx 
\end{multline}
\end{corollary}
To obtain error identities, we need to express \eqref{compound_F_double} for two particular cases where $\yz=\pz$ and  $v=u$. For the first case, we have
\ben
\label{compound_F_primal_double}
D_{\wh F}(v_0,-\Lambda^*\pz)\!
=\! \int\limits_{\Ovp}\!\!\left(\alpha_{\op}-( \dvg \pz+f)\right)v \dx + \int\limits_{\Ovm}\!\!\left(-\alpha_{\om}-( \dvg \pz+f)\right)v \dx.
\een
Since $\pz= A\nabla u$ the relation \eqref{equilibrium_double} guarantees that
$ \dvg \pz +f \in [- \alpha_{\om}, - \alpha_{\op}]$ 
almost everywhere in $\Omega$ and, therefore, $\pz\in  Y^*_{\dvg,[\alpha_-,\alpha_+]} $.
Introduce the sets
\be
 \omega_+:=\Ovp \cap \Oun, \quad
 \omega_-:=\Ovm \cap \Oun, \quad
\omega_\pm:=\left\{\Ovp \cap \Oum \right\}\cup \left\{\Ovm \cap \Oup \right\},
\ee
which qualify the difference between exact coincidence sets and those
formed by $v$ (see Fig. \ref{coincidence_sets_double_obstacle}). The remaining part $\wh \Omega:=\Omega\setminus\omega$ (where $\omega:= \omega_+ \cup \omega_- \cup \omega_\pm$) contains the points of $\Omega$ which 
belong to $\Ovp \cap \Oup$ or $\Ovn \cap \Oun$. In view of (\ref{equilibrium_double}), at these points integrands of (\ref{compound_F_primal_double}) vanish and we obtain
\ben 
\label{measure_primal_double}
 \mu_\omega(v) :=  D_{\wh F}(v_0,-\Lambda^*\pz)=\int\limits_\Omega \alpha(x)|v| \dx,\quad v=v_0+u_D,
\een
where
\ben
\alpha(x)=\left\{
\begin{array}{ll}
\alpha(x)=0\quad & {\rm if}\;x\in \wh \Omega,\\
\alpha(x)=\alpha_\op\quad & {\rm if}\;x\in \,\omega_+,\\
\alpha(x)=\alpha_\om\quad & {\rm if}\;x\in \,\omega_-,\\
\alpha(x)=\alpha_\op + \alpha_\om \quad & {\rm if}\;x\in \,\omega_\pm.
\end{array}
\right.
\een
The right hand side of \eqref {measure_primal_double}  is a nonnegative functional (measure),
which is equal to zero if $\Ovp$ coincides with $\Oup$ and $\Ovm$ coincides
with $\Oum$.
\begin{figure}
  \centering
\begin{minipage}{12cm}
  \includegraphics[width=\textwidth]{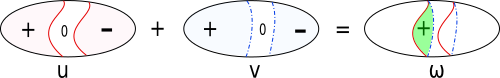} 
\end{minipage} 
\label{coincidence_sets_double_obstacle}
\caption{Illustration example to the double obstacle problem: exact coincidence sets $\Oup, \Oun, \Oum$ with the exact free boundary indicated by red full curves (left), approximate coincidence sets $\Ovp, \Ovn, \Ovm$ with the approximate free boundary indicated by blue dash-dot curves  (middle). Exact and approximate sets do not coincide and it results in an nonempty intersection set 
$\omega_+:=\Ovp \cap \Oun$ filled by the green area (right), where the primal nonlinear measure $\mu_\omega(v)$ contributes to the error.  Note that sets $\omega_-$ and $\omega_\pm$ are empty in this simple example.}
\end{figure}

For the second case, we have
\ben
\label{measure_dual_double}
\mu^*_\omega(\yz) := D_{\wh F}(u_0,-\Lambda^*\yz) = \int\limits_{\Oup}\left(\alpha_{\op}-( \dvg \yz+f)\right)u \dx + \int\limits_{\Oum}\left(-\alpha_{\om}-( \dvg \yz+f)\right)u \dx.
\een
Again we may view the right hand side as a certain measure, which is
zero if  the sets
\be
\Omega^{\yz}_-:=\left\{\dvg\yz+f+\alpha_-=0\right\} \qquad
{\rm and}\qquad
\Omega^{\yz}_+:=\left\{\dvg\yz+f-\alpha_+=0\right\}
\ee
coincide with the sets $\Oum$ and $\Oup$, respectively.

Now  \eqref{erroridentity1}, \eqref{mainidentity}, and  \eqref{mainidentity_dual} 
imply the following result.

\begin{theorem}\label{theorem3}
Let $v\in V_0+u_D$ and $\yz\in Y^*_{\dvg,[\alpha_-,\alpha_+]} $ 
be approximations of $u$ and $\yz$, respectively. Then
\ben
&&\MM (u,v):=\frac12 || \nabla (u - v) ||_A^2 + \mu_\omega(v) = J(v) - J(u), \label{MM_final_obstacle_double} \\
&&\MM (\pz,\yz):=\frac12 ||  \pz - \yz ||_{A^{-1}}^2 +  \mu_{\omega}^{*}(\yz) =\wh I^*(\pz) - \wh I^*(\yz),
 \label{MM_final_obstacle_dual_double}
 \\
\label{equality_double_main}
&&\MM(\{u,\pz\},\{v,\yz\})=\frac12\|A\nabla v-\yz\|^2_{A^{-1}}+ \Upsilon(v,\yz),
\een
where
\ben \label{majorant_nonlinear_double}
\Upsilon(v,\yz):=\IntO(\alpha_+(v)_+
+\alpha_-(v)_- -(f+\dvg\yz) v)\,dx
\een
is a nonnegative functional, which vanishes
if $\yz=\pz$ and $v=u$.
\end{theorem}
\begin{proof}
We apply \eqref{mainidentity} and  \eqref{mainidentity_dual}.
Notice that $\wh J(v_0)=G(\nabla v_0)+F(v_0)=J(v)$.
Next,
\be
D_{\wh F}(v_0,-\Lambda^*\pz)+D_{\wh G}(\Lambda v_0,\pz)=\mu_\omega(v)+
\frac12\|A\nabla(u- v)\|^2_{A}.
\ee
It is easy to see that for any $v=v_0+u_D\in V_0+u_D$, the functional
\be
J(v):=\frac12\IntO A\nabla v\cdot\nabla v\,dx-\IntO(fv-\alpha_+(v)_+-\alpha_(v)_-)dx
\ee
 coincides with $\wh J(v_0)$ and $\wh J(u_0)$ coincides with $J(u)$.
 Since
 \be
 D_{\wh F}(v_0,-\Lambda^*\pz)+D_{\wh G}(\Lambda v_0,\pz)=
 \wh J(v_0)-\wh J(u_0)=J(v)-J(u),
 \ee
we arrive at (\ref{MM_final_obstacle_double}).

Since $u_0=u-u_D$ (where $u$ satisfies the relation $A\nabla u=\pz$), we use (\ref{DwhG}) and \eqref{measure_dual_double}
and obtain
\be
D_{\wh F}(u_0,-\Lambda^*\yz)+
D_{\wh G}(\Lambda u_0,\yz)=\mu^*_\omega(\yz)+\frac12 ||  \pz - \yz ||_{A^{-1}}^2.
\ee
Now (\ref{mainidentity_dual}) yields  
(\ref{MM_final_obstacle_dual_double}), where
\ben
\wh I^*(\yz)=-\wh G^*(\yz)-\wh F^*(-\Lambda^*\yz)=-\frac12\|\yz\|^2_{A^{-1}}
+\IntO (\yz\cdot \nabla u_D+\dvg\yz u_D) \dx\\
=-\frac12\|\yz\|^2_{A^{-1}}
+\int\limits_{\partial\Omega} (\yz\cdot  n) u_D \dx.
\een

Finally, summation of (\ref{MM_final_obstacle_double}) and (\ref{MM_final_obstacle_dual_double}) yields
\ben
\MM(\{u,\pz\},\{v,\yz\})
 = \wh J(v_0) - \wh I^*(\yz)=
 J(v) - \wh I^*(\yz)= \frac12\|A\nabla v-\yz\|^2_{A^{-1}}+ \Upsilon(v,\yz),
\label{MM_final_obstacle_primal_dual_double} 
\een
where
\be
\Upsilon(v,\yz)=\IntO(\alpha_+(v)_++\alpha_-(v)_- -fv +\yz\cdot \nabla (v-u_D)-\dvg\yz u_D) \dx 
=\IntO(\alpha_+(v)_+
+\alpha_-(v)_- -(f+\dvg\yz) v) \dx.
\ee
\end{proof}
\begin{corollary}
From (\ref{equality_double_main}) it follows that
\ben
\label{theorem_double_corollary}
\frac12 || \nabla (u - v) ||_A^2 + 
\frac12 ||  \pz - \yz ||_{A^{-1}}^2 \,\leq\,
\frac12\|A\nabla v-\yz\|^2_{A^{-1}}+ \Upsilon(v,\yz).
\een
This inequality has a practical value because it provides a directly computable
upper bound of the error.
\end{corollary}
\begin{remark}
It is not difficult to show that
$\Upsilon(v,\yz)=0$ if and only if the set
$\Omega^{\yz}_{<>}:=\Omega\setminus\left\{\Omega^{\yz}_-\cup \Omega^{\yz}_+\right\}$ (in this set $-\alpha_-<\dvg \yz+f<\alpha_+$) is a subset of $\Omega^v_0$,
$v$ does not have positive values in $\Omega^{\yz}_-$
and negative values in $\Omega^{\yz}_+$.
To prove this
we represent $\Upsilon(v,\yz)$ in the form
\begin{multline*}
\Upsilon(v,\yz)=\int\limits_{\Omega^{\yz}_-}
(\alpha_+(v)_+
+\alpha_-((v)_- +v) \dx
+\int\limits_{\Omega^{\yz}_{<>}}
(\alpha_+(v)_+
+\alpha_-(v)_- -(f+\dvg\yz) v) \dx
\\
+
\int\limits_{\Omega^{\yz}_+}
(\alpha_+((v)_+-v)
+\alpha_-(v)_-) \dx\\
=\int\limits_{\Omega^{\yz}_-}
(\alpha_++\alpha_-)(v)_+ \dx
+\int\limits_{\Omega^{\yz}_{<>}}
(\alpha_+\chi_{\{v>0\}}-\alpha_-\chi_{\{v>0\}}-f-\dvg\yz) v \dx\\+
\int\limits_{\Omega^{\yz}_+}
(\alpha_++\alpha_-)(v)_- \dx
=\Upsilon_1(v,\yz)+\Upsilon_2(v,\yz)+\Upsilon_3(v,\yz),
\end{multline*}
where the terms are defined by the relations
\be
&&\Upsilon_1(v,\yz)=
\int\limits_{\{\Omega^{\yz}_-\cap \Omega^v_+\}\cup
\{\Omega^{\yz}_+\cap \Omega^v_-\}}
(\alpha_++\alpha_-)| v | \dx,\\
&&\Upsilon_2(v,\yz)=
\int\limits_{\Omega^{\yz}_{<>}\cap\Omega^v_+} W_+(\yz)| v| \dx,\qquad\Upsilon_3(v,\yz)=
\int\limits_{\Omega^{\yz}_{<>}\cap\Omega^v_-}
 W_-(\yz) |v| \dx,
\ee
with the weights
$W_+(\yz)=
(\alpha_+-f-\dvg\yz) $ and  $W_-(\yz)=\alpha_-+f+\dvg\yz$.
The term $\Upsilon_1(v,\yz)$ vanishes if 
$v\leq 0$ in $\Omega^{\yz}_-$ and
$v\geq 0$ in $\Omega^{\yz}_+$. In the set $\Omega^{\yz}_{<>}$ the weights
$W_+(\yz)$ and $W_-(\yz)$ are positive. Therefore, 
$\Upsilon_2(v,\yz)=\Upsilon_3(v,\yz)=0$ implies
$v=0$
almost everywhere
in
$\Omega^{\yz}_{<>}$, i.e., $\Omega^{\yz}_{<>}\subset\Omega^v_0$.  
If all the above conditions are satisfied, then $\Upsilon(v,\yz)=0$
and we arrive at the identity
\ben
\label{simplifiedidentity}
\MM(\{u,\pz\},\{v,\yz\})=\frac12\|A\nabla v-\yz\|^2_{A^{-1}}.
\een

It is clear that $\Upsilon(v,\yz)=0$ if the set $\Omega^v_-$ coincides
(up to a set of zero measure) with the set  $\Omega^{\yz}_-$ and 
$\Omega^v_+$ coincides
 $\Omega^{\yz}_+$. 
\end{remark}
\begin{remark}
Computable upper bound of the primal error measure
$\MM(u, v)$  was first derived in
\cite{ReVa2015}. It has the form
\begin{multline}
\MM (u,v) \,\leq\, \MM^+(v; \beta, \lambda_+, \lambda_-, \yz) \\
:=\frac{1}{2}     (1+\beta) || A \nabla v - \yz ||_{\Omega,A^{-1}}^2 
+ \frac{1}{2}   (1+\frac{1}{\beta})  C_{\Omega}^2 || \dvg \yz + f - \alpha_+ \lambda_+ + \alpha_- \lambda_- ||_{\Omega}^2  \\  + \IntO \Big( \alpha_+ \big( v_+ - \lambda_+ v \big) + \alpha_- \big(v_- + \lambda_- v\big)\Big) \dx.
\end{multline}
The majorant $\MM^+$ contains contains free variables:  $\beta > 0$,  $\yz\in Y^*_{\dvg}(\Omega)$, and two nonnegative functions (Lagrange multipliers) $\lambda_+, \lambda_- \in L^2(\Omega)$ satisfying $\lambda_+(x), \lambda_-(x) \in [0,1]$ almost for all $x \in \Omega$. The constant $C_\Omega > 0$ is given by \eqref{friedrichs}.
In practical computations \cite{BoVa} it is convenient to simplify $\MM^+(v; \beta, \lambda_+, \lambda_-, \yz)$ to
\ben
\qquad \MM_1^+(v; \beta, \lambda, \yz) 
:=\frac{1}{2}     (1+\beta) || A \nabla v - \yz ||_{\Omega,A^{-1}}^2 + \frac{1}{2}   (1+\frac{1}{\beta})  C_{\Omega}^2 || \dvg \yz + f - \lambda ||_{\Omega}^2   + \IntO ( \alpha_+ v_+ + \alpha_- v_- - \lambda v) \dx,
\een
where only one multiplier $\lambda \in L^2(\Omega)$ satisfying $\lambda \in [-\alpha_-,\alpha_+]$ almost for all $x \in \Omega$ is required. 
\end{remark}

\section{Numerical verifications of the error identities}

\subsection{The classical obstacle problem}
We consider an example from \cite{HaVa} with known exact solution. Here, 
$$\Omega=(0, 1), \qquad f=const<0, \qquad A=1, \qquad \phi=const<0$$
and $u$ satisfies the homogeneous Dirichlet boundary conditions $u(0)=0, u(1)=0.$ The exact solution $u$ is in the form 
\begin{equation}
u(x)= u_{f,\phi}(x)=\left\{
\begin{array}{lrl}
\medskip -\frac{f}{2}x^{2}-\sqrt{2 f \phi} \, x & \quad\textrm{if} & x\in [0,\frac{1}{2}-r] \\
\medskip \phi & \quad\textrm{if} & x \in (\frac{1}{2}-r,\frac{1}{2}+r) \\
\medskip -\frac{f}{2}(x-1)^{2}+\sqrt{2 f \phi} \,(x-1) & \quad\textrm{if} & x\in [\frac{1}{2}+r,1]
\end{array} \right.
\end{equation}
where $r=r_{f,\phi}:=\frac{1}{2}-\sqrt{\frac{2\phi}{f}} \in (0, \frac12)$. The parameter $r$ determines the radius of the exact lower coincidence set 
$$\Oum =(\frac{1}{2}-r, \frac{1}{2}+r).$$
It is easy to show that the exact energy reads 
$$J(u)=f  \phi (\frac{4}{3} \sqrt{\frac{2\phi}{f}}-1).$$

An approximation $v_{\epsilon_1}$ is considered in the form 
of $u$ corresponding to the same value of $\phi$ and a perturbed value $f$,
$$ v_{\epsilon_1}(x):= u_{f_{\epsilon_1},\phi}(x), \qquad f_{\epsilon_1}:=\frac{2 \phi}{(\frac12 - r + \epsilon_1)^2} $$
for some small perturbation $\epsilon_1$. This choice ensures 
$$\Omega ^{v_{\epsilon_1}}_\om  =(\frac12-r+\epsilon_1, \frac{1}{2}+r-\epsilon_1) \qquad \mbox{for }
\epsilon_1 \in (r - \frac12, r)
$$
and in particular, $ \Omega ^{v_{\epsilon_1}}_\om \subset \Oum $ for $\epsilon_1 \in (0, r)$. An example of $u$ and $v_{\epsilon_1}$ is depicted in the top left picture of Figure \ref{figure_classical_obstacle}. An approximation $\yz_{\epsilon_2}$ is taken as 
$$ \yz_{\epsilon_2}(x) = I(\pz)(x), \qquad x \in \Omega,$$
where $I$ denotes a piecewise linear nodal and continuous interpolation operator at nodes
$$\{0, \frac12 - r-\epsilon_2, \frac12 - r + \epsilon_2, \frac12 + r -\epsilon_2, \frac12 + r + \epsilon_2, 1 \}$$ for some small positive perturbation $\epsilon_2$. The approximation $\yz_{\epsilon_2}$ differs from the exact flux $\pz$ only locally in $(\frac12 - r - \epsilon_2, \frac12 - r + \epsilon_2) \cup (\frac12 + r - \epsilon_2, \frac12 + r + \epsilon_2) $ and
$$\Oum \subset \Omega ^{\yz_{\epsilon_2}}_\om  = (\frac12 - r-\epsilon_2, \frac12 + r + \epsilon_2) \quad \mbox{for } \epsilon_2 \in (0, r).$$
An example of $\pz$ and $\yz_{\epsilon_2}$ is shown in the top right picture of Figure \ref{figure_classical_obstacle} and corresponding equilibrium terms $\dvg \pz + f$ and $\dvg \yz_{\epsilon_2}+ f$ in the bottom left picture. \\

For numerical verifications, we choose parameters 
$$\phi=-1, \quad f=-14$$ 
resulting in $r \approx 0.1220$, $J(u)\approx -6.9446$
and approximations $v_{\epsilon_1}, \yz_{\epsilon_2}$ corresponding to 
$$\epsilon_1, \epsilon_2 \in \{1/10, 1/20, 1/40, 1/80, 1/160 \}.$$
We first verify the primal error identity
$$ \frac{1}{2}\|\nabla(u-v_{\epsilon_1})\|^{2} + \mu_{\phi\psi}(v_{\epsilon_1}) = \MM(u, v_{\epsilon_1}) = J(v_{\epsilon_1})-J(u) $$ for all approximations $v_{\epsilon_1}$.
Table \ref{table_classical_obstacle_1D_primal} confirms that the primal error identity holds and both quadratic (gradient containing) and nonlinear parts of the primal error converge. For smaller values of $\epsilon_1$ the quadratic part dominates over the nonlinear part. This is due to the fact that the quadratic part of error is globally distributed over $\Omega$ and the nonlinear part $\mu_{\phi\psi}(v_{\epsilon_1})$ has a support in 
\be
\Oum \setminus \Omega ^{v_{\epsilon_1}}_\om \approx (0.3779, 0.3779+ \epsilon_1) \, \cup \, (0.6220- \epsilon_1, 0.6220).
\ee
Table \ref{table_classical_obstacle_1D_dual} verifies the dual error identity
$$ \frac{1}{2}\|\pz - \yz_{\epsilon_2}\|^{2} + \mu_{\phi\psi}^{*}(\yz_{\epsilon_2}) = \MM(\pz,\yz_{\epsilon_2}) = I^*(\pz) - I^*(\yz_{\epsilon_2}) $$ 
for all approximations $\yz_{\epsilon_2}$. Again, both quadratic and nonlinear parts converge. None of error parts dominates, since $\yz_{\epsilon_2}$ and $\pz_{\epsilon_2}$ differ only locally. The nonlinear part $\mu_{\phi\psi}^{*}(\yz_{\epsilon_2})$ has a support in 
\be
\Omega ^{\yz_{\epsilon_2}}_\om \setminus \Oum \approx (0.3779 - \epsilon_2, 0.3779) \, \cup \, (0.6220, 0.6220 + \epsilon_2).
\ee
An example of primal and dual nonlinear error functions is depicted in the bottom left picture of of Figure \ref{figure_classical_obstacle}. \\

Table \ref{table_classical_obstacle_1D_majorant} verifies the majorant identity
$$\MM(\{u,\pz\},\{v_{\epsilon_1},\yz_{\epsilon_2}\})=\frac12 ||A \nabla v_{\epsilon_1} - \yz_{\epsilon_2} ||^2 + \Upsilon(v_{\epsilon_1},\yz_{\epsilon_2}), $$
where the computable nonlinear majorant part  $\Upsilon$  is given by \eqref{majorant_nonlinear_classical}. The majorant identity is valid for all considered approximations. 

\begin{remark}
Since the upper obstacle $\psi$ is not considered in this example, 
$$ \dvg \yz + f \leq 0 \qquad \mbox{in } \Omega$$
has to be satisfied. This condition is fulfilled for $\yz_{\epsilon_2}$ constructed above. 
\end{remark}

\subsection{The double obstacle problem}
We consider an example with known exact solution. Here,  
$$\Omega=(-1, 1), \qquad f=0, \qquad A=1, \qquad \alphap,\alpham >2$$
and $u$ satisfies Dirichlet boundary conditions $u(-1)=-1, u(1)=1.$ This example generalizes example of \cite{Bo}, in which $\alphap = \alpham = 8$.  It is possible to show the exact solution is given by a formula
$$ u(x)=u_{\alpham,\alphap}(x)=\left\{\begin{array}{ll} -(\frac{\alpham}{2})\,x^2 + (\sqrt{2\alpham}-\alpham)\,x + \sqrt{2\alpham}-\frac{\alpham}{2}-1 ,  & x \in \left<-1, r_\om \right>, \\
                                 0,  & x \in ( r_\om, r_\op ), \\
                                 (\frac{\alphap}{2})\,x^2 + (\sqrt{2\alphap}-\alphap)\,x - \sqrt{2\alphap}+\frac{\alphap}{2}+1,  & x \in \left<r_\op, 1 \right>,
                                    \end{array} \right.          $$
where $r_\om: = \sqrt{\frac{2}{\alpham}}-1 \in (-1,0)$ and $r_\op := 1 - \sqrt{\frac{2}{\alphap}} \in (0,1)$ determine exact coincidence sets 
$$ \Oum=(-1, r_\om), \qquad \Oun=\left< r_\om, r_\op \right>, \qquad \Oup=(r_\op, 1).$$
The exact energy then reads
$$ J(u)=\frac{2 \sqrt{2}}{3} (\sqrt{\alphap} + \sqrt{\alpham}). $$ 

An approximation $v_{\epsilon_1}$ is considered in the form of $u$ corresponding to perturbed values $\alphap$, $\alpham$, 
$$ v_{\epsilon_1}(x):= u_{\alpha_{\om_{\epsilon_1}},\alpha_{\op_{\epsilon_1}}}, \qquad 
\alpha_{\opm_{\epsilon_1}}:=\frac{2}{(1 \mp r_\opm + \epsilon_1)^2} 
 $$
for some small perturbation $\epsilon_1$. This choice ensures 
$$\Omega ^{v_{\epsilon_1}}_0  =(r_\om +\epsilon_1, r_\op - \epsilon_1) \qquad \mbox{for } 
\epsilon_1 \in (-1 + r_\op, r_\op) \cap (-1 - r_\om, -r_\om) $$
and in particular, $ \Omega ^{v_{\epsilon_1}}_0 \subset \Oun $ for $\epsilon_1 \in (0, \min \{r_\op, -r_\om \})$. An example of $u$ and $v_{\epsilon_1}$ is depicted in the top left picture of Figure \ref{figure_double_obstacle}. An approximation $\yz_{\epsilon_2}$ is taken as 
$$ \yz_{\epsilon_2}(x) = I(\pz)(x), \qquad x \in \Omega,$$
where $I$ denotes a piecewise linear nodal and continuous interpolation operator at nodes
$$\{-1, r_\om -\epsilon_2, r_\om +\epsilon_2, r_\op -\epsilon_2, r_\op + \epsilon_2, 1 \}$$ for some small positive perturbation $\epsilon_2$. The approximation $\yz_{\epsilon_2}$ differs from the exact flux $\pz$ only locally in $(r_\om  - \epsilon_2, r_\om  + \epsilon_2) \cup (r_\op  - \epsilon_2, r_\op  + \epsilon_2) $ and 
$$\Oum \subset \Omega ^{\yz_{\epsilon_2}}_\om  = (r_\om -\epsilon_2, r_\op  + \epsilon_2) \quad \mbox{for } \epsilon_2 \in (0, \min \{ -r_\om, r_\op   \}).$$
An example of $\pz$ and $\yz_{\epsilon_2}$ is shown in the top right picture of Figure \ref{figure_double_obstacle} and corresponding equilibrium terms in the bottom left picture. \\

For numerical verifications, we choose parameters (identical to example of \cite{Bo})
$$\alpham=\alphap=8 $$ 
resulting in 
$$ u(x)=\left\{\begin{array}{ll} -4x^2-4x-1 , \quad & x \in \left<-1, -0.5 \right>, \\
                                 0, \quad & x \in (-0.5, 0.5 ), \\
                                 4x^2-4x+1, \quad & x \in \left<0.5, 1 \right>
                                    \end{array} \right. $$
and $J(u)=5\frac{1}{3}$ and approximations $v_{\epsilon_1}, \yz_{\epsilon_2}$ corresponding to 
$$\epsilon_1, \epsilon_2 \in \{1/5, 1/10, 1/20, 1/40, 1/80 \}.$$
We again verify the primal error identity
$$ \frac{1}{2}\|\nabla(u-v_{\epsilon_1})\|^{2} + \mu_{\omega}(v_{\epsilon_1}) = \MM(u, v_{\epsilon_1}) = J(v_{\epsilon_1})-J(u) $$ for all approximations $v_{\epsilon_1}$.
Table \ref{table_double_obstacle_1D_primal} confirms that the primal error identity holds and both quadratic (gradient containing) and nonlinear parts of the primal error converge. For smaller values of $\epsilon_1$ the quadratic part dominates over the nonlinear part. This is due to the fact that the quadratic part of error is globally distributed over $\Omega$ and the nonlinear part $\mu_{\phi\psi}(v_{\epsilon_1})$ has a support in 
\be
\Oum \setminus \Omega ^{v_{\epsilon_1}}_\om \approx (-0.5, -0.5+ \epsilon_1) \, \cup \, (0.5-\epsilon_1, 0.5).
\ee
Table \ref{table_double_obstacle_1D_dual} verifies the dual error identity
$$ \frac{1}{2}\|\pz - \yz_{\epsilon_2}\|^{2} + \mu_{\omega}^{*}(\yz_{\epsilon_2}) = \MM(\pz,\yz_{\epsilon_2}) = I^*(\pz) - I^*(\yz_{\epsilon_2}) $$ 
for all approximations $\yz_{\epsilon_2}$. Again, both quadratic and nonlinear parts converge. None of error parts dominates, since $\yz_{\epsilon_2}$ and $\pz_{\epsilon_2}$ differ only locally. The nonlinear part $\mu_{\phi\psi}^{*}(\yz_{\epsilon_2})$ has a support in 
\be
\Omega ^{\yz_{\epsilon_2}}_\om \setminus \Oum \approx (0.5 - \epsilon_2, 0.5) \, \cup \, (0.5, 0.5 + \epsilon_2).
\ee
An example of primal and dual nonlinear error functions is depicted in the bottom left picture of of Figure \ref{figure_double_obstacle}. \\

Table \ref{table_double_obstacle_1D_majorant} verifies the majorant identity
$$\MM(\{u,\pz\},\{v_{\epsilon_1},\yz_{\epsilon_2}\})=\frac12 ||A \nabla v_{\epsilon_1} - \yz_{\epsilon_2} ||^2 + \Upsilon(v_{\epsilon_1},\yz_{\epsilon_2}), $$
where the computable nonlinear majorant part  $\Upsilon$  is given by \eqref{majorant_nonlinear_double}. The majorant identity is valid for all considered approximations. 

\begin{table}
\begin{center}
\begin{small}
\begin{tabular}{|l|c|c||c|c||c|}
\hline
$\epsilon_1$   & $\frac{1}{2}\|\nabla(u-v_{\epsilon_1})\|^{2}$ & $\mu_{\phi\psi}(v_{\epsilon_1}) $ & $\MM(u, v_{\epsilon_1})$ & $J(v_{\epsilon_1})-J(u)$  & $\mu_{\phi\psi}(v_{\epsilon_1})$ [\%] \\
\hline
0.1000  & 1.54e-01 & 4.09e-02 & 1.95e-01 & 1.95e-01 & 20.92 \\ 
0.0500  & 4.82e-02 & 6.37e-03 & 5.45e-02 & 5.45e-02 & 11.68 \\ 
0.0250  & 1.36e-02 & 8.98e-04 & 1.45e-02 & 1.45e-02 & 6.20 \\ 
0.0125  & 3.62e-03 & 1.20e-04 & 3.73e-03 & 3.73e-03 & 3.20 \\ 
0.0063  & 9.33e-04 & 1.54e-05 & 9.49e-04 & 9.49e-04 & 1.63 \\ 
\hline
\end{tabular}
\caption{Terms in the primal error identity computed for various approximation $v_{\epsilon_1}$ in case of the classical obstacle in 1D. The rightmost column shows the contribution of $\mu_{\phi\psi}(v_{\epsilon_1})$ to the majorant value $\MM(u, v_{\epsilon_1})$. }
\label{table_classical_obstacle_1D_primal}
\begin{tabular}{|l|c|c||c|c||c|}
\hline
$\epsilon_2$   & $\frac{1}{2}\|\pz - \yz_{\epsilon_2}\|^{2}$ & $\mu_{\phi\psi}^{*}(\yz_{\epsilon_2}) $ & $\MM(\pz,\yz_{\epsilon_2})$ & $ I^*(\pz) - I^*(\yz_{\epsilon_2}) $  & $\mu_{\phi\psi}^{*}(\yz_{\epsilon_2})$ [\%] \\
\hline
0.0500  & 4.08e-03 & 4.08e-03 & 8.17e-03 & 8.17e-03 & 50.00 \\ 
0.0250  & 5.10e-04 & 5.10e-04 & 1.02e-03 & 1.02e-03 & 50.00 \\ 
0.0125  & 6.38e-05 & 6.38e-05 & 1.28e-04 & 1.28e-04 & 50.00 \\ 
0.0063  & 7.98e-06 & 7.98e-06 & 1.60e-05 & 1.60e-05 & 50.00 \\ 
0.0031  & 9.97e-07 & 9.97e-07 & 1.99e-06 & 1.99e-06 & 50.00 \\ 
\hline
\end{tabular}
\caption{Terms in the dual error identity computed for various approximation $\yz_{\epsilon_2}$  in case of the classical obstacle in 1D.  The rightmost column shows the contribution of $\mu_{\phi\psi}^{*}(\yz_{\epsilon_2}) $ to the majorant value $\MM(\pz,\yz_{\epsilon_2})$.}
\label{table_classical_obstacle_1D_dual}
\begin{tabular}{|l|l|c|c||c|c||}
\hline
$\epsilon_1$   & $\epsilon_2$   &  $\frac12 ||\nabla v_{\epsilon_1} - \yz_{\epsilon_2} ||^2$  &  $\Upsilon(v_{\epsilon_1},\yz_{\epsilon_2})$    &sum & $\MM(\{u,\pz\},\{v_{\epsilon_1},\yz_{\epsilon_2}\})$ \\
\hline
0.1000  & 0.1000 & 9.72e-02 & 1.63e-01 & 2.61e-01 & 2.61e-01 \\ 
0.1000  & 0.0500 & 1.32e-01 & 7.15e-02 & 2.03e-01 & 2.03e-01 \\ 
0.0500  & 0.0500 & 3.72e-02 & 2.55e-02 & 6.27e-02 & 6.27e-02 \\ 
0.0500  & 0.0250 & 4.44e-02 & 1.11e-02 & 5.55e-02 & 5.55e-02 \\ 
0.0250  & 0.0250 & 1.19e-02 & 3.59e-03 & 1.55e-02 & 1.55e-02 \\ 
0.0250  & 0.0125 & 1.30e-02 & 1.57e-03 & 1.46e-02 & 1.46e-02 \\ 
0.0125  & 0.0125 & 3.38e-03 & 4.78e-04 & 3.86e-03 & 3.86e-03 \\ 
0.0125  & 0.0063 & 3.54e-03 & 2.09e-04 & 3.75e-03 & 3.75e-03 \\ 
0.0063  & 0.0063 & 9.03e-04 & 6.17e-05 & 9.65e-04 & 9.65e-04 \\ 
0.0063  & 0.0031 & 9.24e-04 & 2.70e-05 & 9.51e-04 & 9.51e-04 \\ 
\hline
\end{tabular}
\caption{Terms in the majorant error identity computed for various approximation $v_{\epsilon_1}$ and $\yz_{\epsilon_2}$ in case of the classical obstacle in 1D. The computable majorant  $\frac12 ||\nabla v_{\epsilon_1} - \yz_{\epsilon_2} ||^2 + \Upsilon(v_{\epsilon_1},\yz_{\epsilon_2})$   is identical to $\MM(\{u,\pz\},\{v_{\epsilon_1},\yz_{\epsilon_2}\})$, which can only be computed with the knowledge of the exact solution $u$ and the exact flux $\pz$. 
}
\label{table_classical_obstacle_1D_majorant}
\end{small}
\end{center}
\end{table}


\begin{figure}
 \centering
 \includegraphics[width=0.8\textwidth]{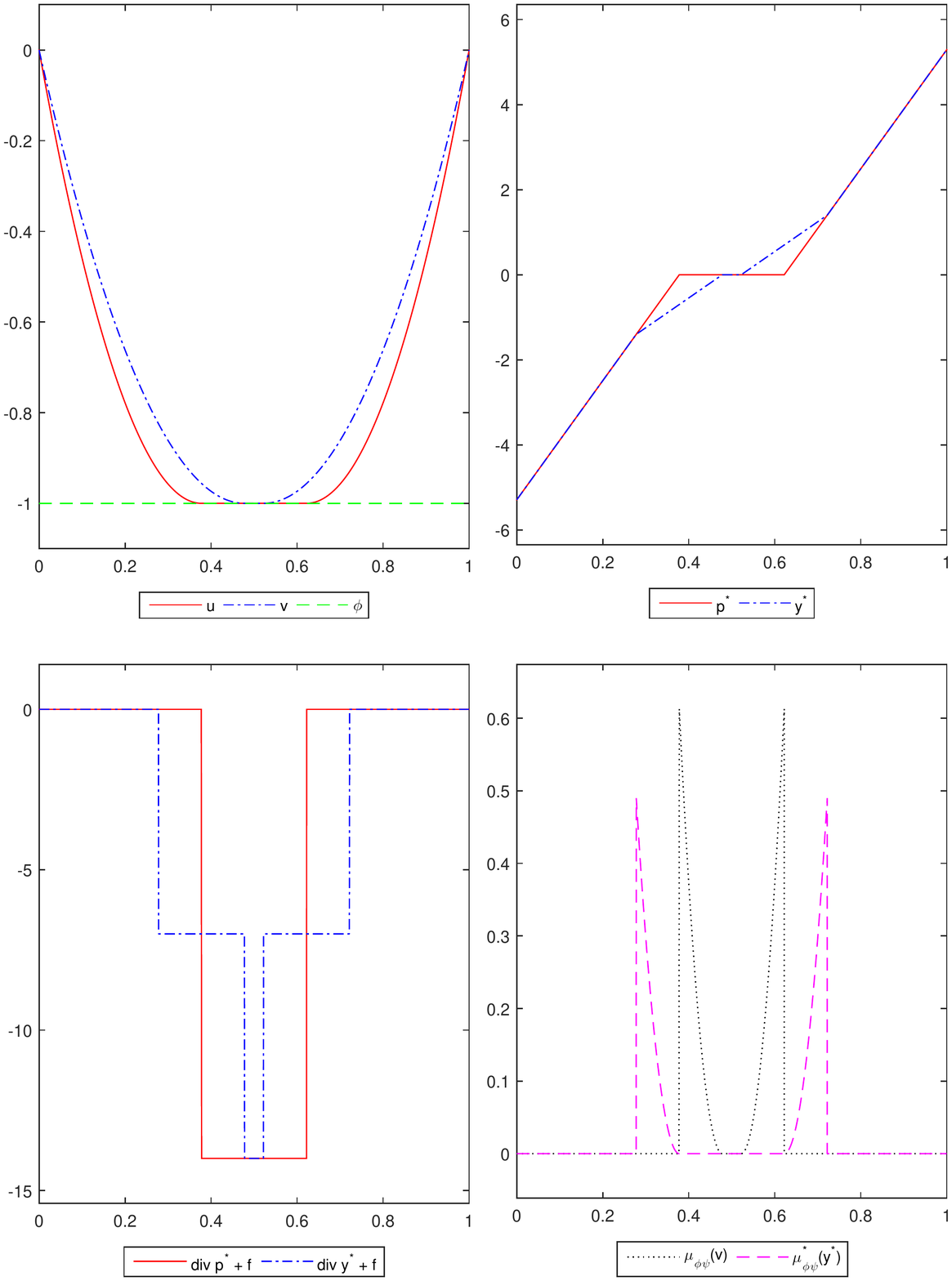}
 \caption{Example of the classical obstacle problem with $\phi=-1$ and $f=-14$ and approximations $v_{\epsilon_1}$ and $\yz_{\epsilon_2}$ generated by perturbations $\epsilon_1=\epsilon_2=0.1.$
The exact solution $u$ and its approximation $v_{\epsilon_1}$ 
are displayed in the top left picture, the exact flux $\pz=u'$ and its approximation $\yz_{\epsilon_2}$ in the top right picture. Both nonpositive functions $\dvg \, \pz + f$ and $\dvg \, \yz_{\epsilon_2} + f$ are displayed in the bottom left picture.   
Since $\Ovm \subset \Oum$ and $\subset \Oum \subset \Omega^{\yz}_\om$ , there are positive contributions of $\mu_{\phi\psi}(v_{\epsilon_1}) $ and $\mu_{\phi\psi}^{*}(\yz_{\epsilon_2})$ shown in the bottom right picture.}
\label{figure_classical_obstacle}
\end{figure}

\begin{table}
\begin{center}
\begin{small}
\begin{tabular}{|l|c|c||c|c||c|}
\hline
$\epsilon_1$   & $\frac{1}{2}\|\nabla(u-v_{\epsilon_1})\|^{2}$ & $\mu_{\omega}(v_{\epsilon_1}) $ & $\MM(u, v_{\epsilon_1})$ & $J(v_{\epsilon_1})-J(u)$  & $\mu_{\omega}(v_{\epsilon_1})$ [\%] \\
\hline
0.2000  & 2.18e-01 & 8.71e-02 & 3.05e-01 & 3.05e-01 & 28.57 \\ 
0.1000  & 7.41e-02 & 1.48e-02 & 8.89e-02 & 8.89e-02 & 16.67 \\ 
0.0500  & 2.20e-02 & 2.20e-03 & 2.42e-02 & 2.42e-02 & 9.09 \\ 
0.0250  & 6.05e-03 & 3.02e-04 & 6.35e-03 & 6.35e-03 & 4.76 \\ 
0.0125  & 1.59e-03 & 3.97e-05 & 1.63e-03 & 1.63e-03 & 2.44 \\ 
\hline
\end{tabular}
\caption{Terms in the primal error identity computed for various approximation $v_{\epsilon_1}$ in case of the double obstacle in 1D. The rightmost column shows the contribution of $\mu_{\phi\psi}(v_{\epsilon_1})$ to the majorant value $\MM(u, v_{\epsilon_1})$. }
\label{table_double_obstacle_1D_primal}
\begin{tabular}{|l|c|c||c|c||c|}
\hline
$\epsilon_2$   & $\frac{1}{2}\|\pz - \yz_{\epsilon_2}\|^{2}$ & $\mu_{\omega}^{*}(\yz_{\epsilon_2}) $ & $\MM(\pz,\yz_{\epsilon_2})$ & $ I^*(\pz) - I^*(\yz_{\epsilon_2}) $  & $\mu_{\omega}^{*}(\yz_{\epsilon_2})$ [\%] \\
\hline
0.2000  & 8.53e-02 & 8.53e-02 & 1.71e-01 & 1.71e-01 & 50.00 \\ 
0.1000  & 1.07e-02 & 1.07e-02 & 2.13e-02 & 2.13e-02 & 50.00 \\ 
0.0500  & 1.33e-03 & 1.33e-03 & 2.67e-03 & 2.67e-03 & 50.00 \\ 
0.0250  & 1.67e-04 & 1.67e-04 & 3.33e-04 & 3.33e-04 & 50.00 \\ 
0.0125  & 2.08e-05 & 2.08e-05 & 4.17e-05 & 4.17e-05 & 50.00 \\ 
\hline
\end{tabular}
\caption{Terms in the dual error identity computed for various approximation $\yz_{\epsilon_2}$  in case of the double obstacle in 1D.  The rightmost column shows the contribution of $\mu_{\phi\psi}^{*}(\yz_{\epsilon_2}) $ to the majorant value $\MM(\pz,\yz_{\epsilon_2})$.}
\label{table_double_obstacle_1D_dual}
\begin{tabular}{|l|l|c|c||c|c||}
\hline
$\epsilon_1$   & $\epsilon_2$   &  $\frac12 ||\nabla v_{\epsilon_1} - \yz_{\epsilon_2} ||^2$  &  $\Upsilon(v_{\epsilon_1},\yz_{\epsilon_2})$    &sum & $\MM(\{u,\pz\},\{v_{\epsilon_1},\yz_{\epsilon_2}\})$ \\
\hline
0.2000  & 0.2000 & 1.27e-01 & 3.48e-01 & 4.75e-01 & 4.75e-01 \\ 
0.1000  & 0.2000 & 5.96e-02 & 2.00e-01 & 2.60e-01 & 2.60e-01 \\ 
0.1000  & 0.1000 & 5.10e-02 & 5.93e-02 & 1.10e-01 & 1.10e-01 \\ 
0.0500  & 0.1000 & 1.58e-02 & 2.98e-02 & 4.56e-02 & 4.56e-02 \\ 
0.0500  & 0.0500 & 1.81e-02 & 8.82e-03 & 2.69e-02 & 2.69e-02 \\ 
0.0250  & 0.0500 & 4.93e-03 & 4.08e-03 & 9.02e-03 & 9.02e-03 \\ 
0.0250  & 0.0250 & 5.47e-03 & 1.21e-03 & 6.68e-03 & 6.68e-03 \\ 
0.0125  & 0.0250 & 1.42e-03 & 5.35e-04 & 1.96e-03 & 1.96e-03 \\ 
0.0125  & 0.0125 & 1.51e-03 & 1.59e-04 & 1.67e-03 & 1.67e-03 \\ 
0.0063  & 0.0125 & 3.85e-04 & 6.86e-05 & 4.53e-04 & 4.53e-04 \\ 
\hline
\end{tabular}
\caption{Terms in the majorant error identity computed for various approximation $v_{\epsilon_1}$ and $\yz_{\epsilon_2}$ in case of the double obstacle in 1D. The computable majorant  
$\frac12 ||\nabla v_{\epsilon_1} - \yz_{\epsilon_2} ||^2 + \Upsilon(v_{\epsilon_1},\yz_{\epsilon_2})$  is identical to $\MM(\{u,\pz\},\{v_{\epsilon_1},\yz_{\epsilon_2}\})$, which can only be computed with the knowledge of the exact solution $u$ and the exact flux $\pz$. 
}
\label{table_double_obstacle_1D_majorant}
\end{small}
\end{center}
\end{table}

\begin{figure}
 \centering
 \includegraphics[width=0.8\textwidth]{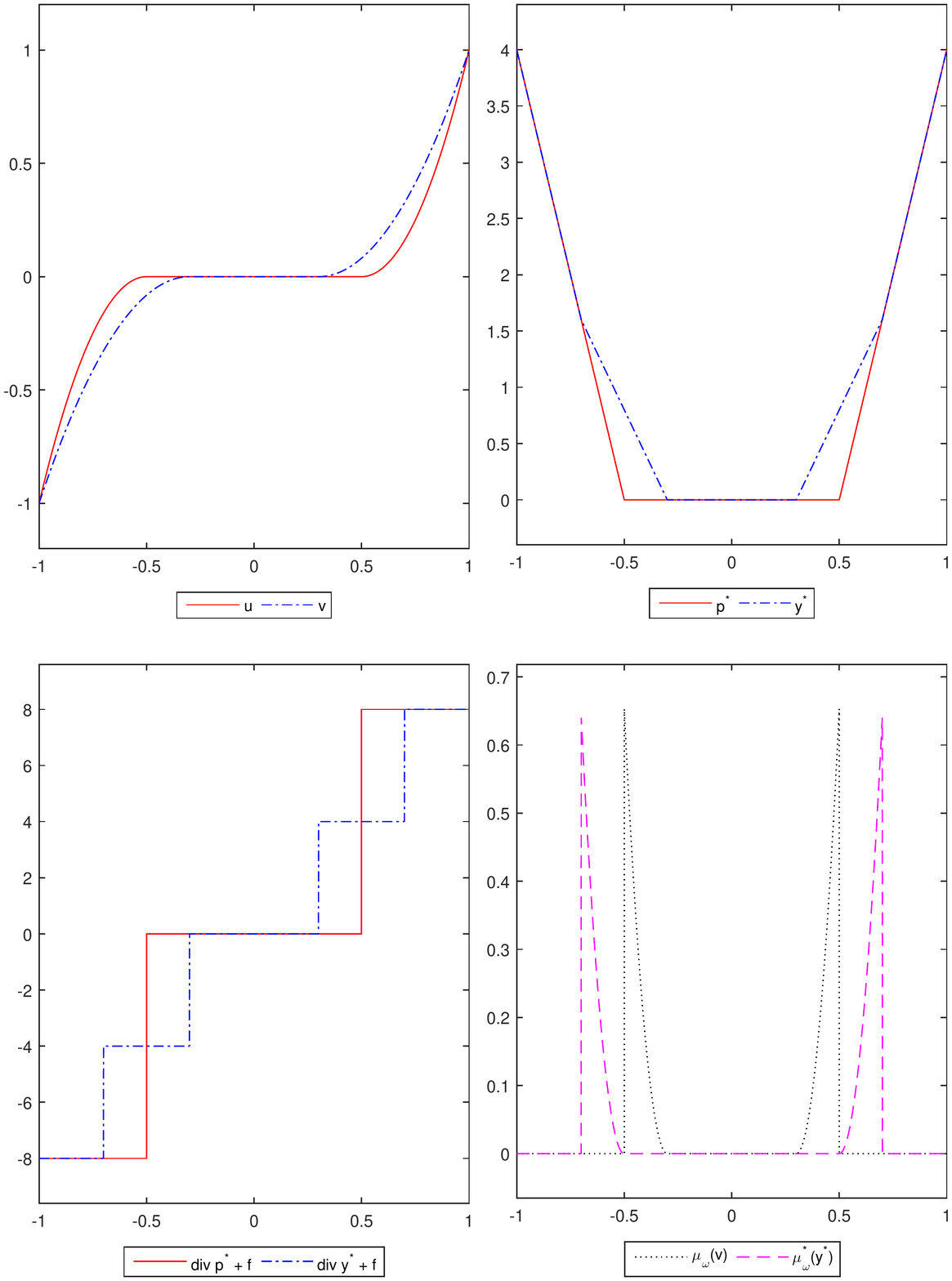}
 \caption{Example of the double obstacle problem with $\alphap=\alpham=8$ and approximations $v_{\epsilon_1}$ and $\yz_{\epsilon_2}$ generated by perturbations $\epsilon_1=\epsilon_2=0.2.$
The exact solution $u$ and its approximation $v_{\epsilon_1}$ 
are displayed in the top left picture, the exact flux $\pz=u'$ and its approximation $\yz_{\epsilon_2}$ in the top right picture. Both nonpositive functions $\dvg \, \pz + f$ and $\dvg \, \yz_{\epsilon_2} + f$ are displayed in the bottom left picture.   
Since $\Ovm \subset \Oum$ and $\subset \Oum \subset \Omega^{\yz}_\om$ , there are positive contributions of $\mu_{\omega}(v_{\epsilon_1}) $ and $\mu_{\omega}^{*}(\yz_{\epsilon_2})$ shown in the bottom right picture.)
\label{figure_double_obstacle}
 }
\end{figure}

\section*{Acknowledgments}
The first author acknowledges the support of the Johann Radon Institute for Computational and Applied Mathematics (RICAM) in Linz, Austria during 
Special Semester on Computational Methods in Science and Engineering in 2016. The second author has been supported by GA CR through the projects
GF16-34894L and 17-04301S. 

\bibliographystyle{siamplain}


\end{document}